\documentclass{amsart}
\usepackage{graphicx}
\usepackage{amsmath}
\usepackage{amssymb}
\usepackage{amsthm}
\usepackage[msc-links,abbrev]{amsrefs}
\usepackage{mathtools}

\makeatletter
\@namedef{subjclassname@2020}{\textup{2020} Mathematics Subject Classification}
\makeatother

\newtheorem{theorem}{Theorem}[section]
\newtheorem{lem}[theorem]{Lemma}

\newtheorem{cor}[theorem]{Corollary}
\newtheorem{prop}[theorem]{Proposition}

\theoremstyle{definition}
\newtheorem{definition}[theorem]{Definition}
\newtheorem{remark}[theorem]{Remark}
\newtheorem{example}[theorem]{Example}

\DeclareMathOperator{\End}{End}
\DeclareMathOperator{\Id}{Id}
\DeclareMathOperator{\Real}{Re}
\DeclareMathOperator{\Imaginary}{Im}
\DeclareMathOperator{\im}{im}

\numberwithin{equation}{section}

\def\sideremark#1{\ifvmode\leavevmode\fi\vadjust{\vbox to0pt{\vss
 \hbox to 0pt{\hskip\hsize\hskip1em
 \vbox{\hsize3cm\tiny\raggedright\pretolerance10000
 \noindent #1\hfill}\hss}\vbox to8pt{\vfil}\vss}}}

\begin{document}

\title[Deformations of MTW scalar curvature]{Deformations of the scalar curvature of a partially integrable pseudohermitian manifold}

\author{Jeffrey S. Case}
\address{Department of Mathematics, Penn State University, University Park, PA 16802, USA}
\email{jscase@psu.edu}


\author{Pak Tung Ho}
\address{Department of Mathematics, Tamkang University, Tamsui, New Taipei City 251301, Taiwan}
\email{paktungho@yahoo.com.hk}

\keywords{partially integrable CR manifold; pseudohermitian manifold; Tanaka--Webster scalar curvature; deformation}
\subjclass[2020]{Primary 58J60; Secondary 32V05 53C21 53D35}

\begin{abstract}
We consider deformations of the scalar curvature of a partially integrable pseudohermitian manifold, in analogy with the work of Fischer and Marsden on Riemannian manifolds.
In particular, we introduce and discuss $R$-singular spaces, give sufficient conditions for the stability of the scalar curvature, and give a partial infinitesimal rigidity result for the scalar curvature of a compact, torsion-free, scalar-flat, integrable pseudohermitian manifold.
\end{abstract}

\maketitle

\section{Introduction}

Fischer and Marsden~\cite{Fischer&Marsden} studied the stability and rigidity of the scalar curvature $R$ of a Riemannian manifold $(M^n,g)$.
Here $R$ is \emph{stable} at $g$ if whenever $h \in \ker DR$, the metric linearization of $R$, there is a path $g(t)$ of metrics with $g(0)=g$ and $g^\prime(0)=h$ such that $R(g(t))$ is constant;
and $R$ is (locally) \emph{rigid} at $g$ if there is a neighborhood $U$ of $g$ in the space of Riemannian metrics such that if $\hat g \in U$ and $R(\hat g) \geq R(g)$, then $\hat g$ is isometric to $g$.
A key tool in these results is the notion of an $R$-singular space;
i.e.\ a Riemannian manifold $(M^n,g)$ for which the kernel of the formal adjoint of $DR$ is nontrivial.
The implicit function theorem implies~\cite{Fischer&Marsden}*{p.\ 519} that if a given compact Riemannian manifold $(M^n,g)$ is not $R$-singular, then $R$ is stable at $g$.
Additionally, one can give necessary conditions for a compact Riemannian manifold to be $R$-singular, and hence sufficient conditions for $R$ to be stable at a given compact Riemannian manifold.
One can also show that $\ker (DR)^\ast = \mathbb{R}$ on compact Ricci-flat manifolds, an important first step in proving the local rigidity of $R$ at such manifolds.

In the Riemannian setting, the key point is that the variational structure of the scalar curvature allows one to derive necessary conditions for $\ker(DR)^\ast \not= \{ 0 \}$.
Lin and Yuan~\cite{Lin&Yaun} illustrated this point by developing stability and rigidity results for the fourth-order $Q$-curvature.
Case, Lin, and Yuan~\cite{CaseLinYuan} completely clarified this point by generalizing these results to \emph{any} variational scalar Riemannian invariant.

In the hermitian setting, Angella and Pediconi~\cite{Angella&Pediconi} proved analogues of the results of Fischer and Marsden for the Chern-scalar curvature.
Again, one can give necessary conditions for a compact hermitian manifold to be such that the adjoint of the linearization of the Chern-scalar curvature has nontrivial kernel.

The purpose of this paper is to initiate the systematic study of stability and rigidity for certain scalar invariants on partially integrable pseudohermitian manifolds by studying the special case of the scalar curvature.
Recall that a (nondegenerate) \emph{partially integrable CR manifold} $(M^{2n+1},J)$ is a contact manifold $(M^{2n+1},\xi)$ together with an almost complex structure $J$ on $\xi$ such that the $(+i)$-eigenspace $T^{1,0}$ of $J$ on $\xi \otimes \mathbb{C}$ satisfies the partial integrability condition
\begin{equation*}
 [ C^\infty(M;T^{1,0}) , C^\infty(M;T^{1,0}) ] \subseteq C^\infty( M ; T^{1,0} \oplus T^{0,1} ) ,
\end{equation*}
where $T^{0,1} := \overline{T^{1,0}}$;
see Section~\ref{bg} for a detailed discussion.
These generalize (nondegenerate) \emph{CR manifolds}, which are the integrable case
\begin{equation*}
 [ C^\infty(M;T^{1,0}) , C^\infty(M;T^{1,0}) ] \subseteq C^\infty(M;T^{1,0}) .
\end{equation*}
Note that every partially integrable CR three-manifold is integrable.

Matsumoto showed~\cite{Matsumoto2014} that if $\theta$ is a contact form on $(M^{2n+1},J)$---that is, if $\theta$ is a real one-form with $\ker\theta = \xi$---then there is a unique connection $\nabla$ which preserves the specified structure and for which the torsion takes a particularly nice form.
Moreover, if $J$ is integrable, then his connection recovers the Tanaka--Webster connection~\cites{Tanaka,Webster}.
Hence we call it the Matsumoto--Tanaka--Webster, or \emph{MTW}, connection.
One then defines the curvature in the usual way.
In this paper, we study the stability and rigidity of the \emph{MTW scalar curvature}.

The reasons to study the MTW scalar curvature, rather than the Tanaka--Webster scalar curvature of an integrable pseudohermitian manifold, are two-fold:

First, infinitesimal deformations of a partially integrable CR manifold are determined algebraically by linearizing the equation $J^2=-1$;
i.e.\ the space of infinitesimal deformations of a partially integrable CR manifold is equivalent to the space of sections of a particular vector bundle.
This is analogous to the situation in Riemannian geometry.
By contrast, infinitesimal deformations of integrable CR manifolds (exception in dimension three) are determined differentially by also linearizing the condition that the Nijenhuis tensor---which is defined on any partially integrable CR manifold~\cite{Matsumoto2014}---vanishes.
Thus the space of infinitesimal deformations of an integrable CR manifold is equivalent to the space of \emph{holomorphic} sections of a certain CR vector bundle~\cite{Akahori}, analogous to the situation in complex geometry.

Second, Matsumoto's study~\cite{Matsumoto2014} of partially integrable CR manifolds reveals many close analogies between such structures and conformal manifolds.
For instance, the $Q$-curvature of a partially integrable CR manifold need not be a divergence, unlike in the integrable case~\cite{Marugame2018};
and the obstruction tensor of a partially integrable CR manifold is generally nontrivial.
Partially integrable CR manifolds, like conformal and CR manifolds, are examples of parabolic geometries~\cite{CapSlovak}.

To describe our main results, regard the MTW scalar curvature $R$ as a $C^\infty(M)$-valued function on the product of the space $\mathcal{C}$ of partially integrable CR structures and the space $\mathcal{K}$ of contact forms on a contact manifold $(M^{2n+1},\xi)$.
Denote its linearization by $DR \colon T_{(J,\theta)}(\mathcal{C} \times \mathcal{K}) \to C^\infty(M)$ and let
\begin{equation*}
 \Gamma \colon C^\infty(M) \to T_{(J,\theta)}(\mathcal{C} \times \mathcal{K})
\end{equation*}
denote its adjoint with respect to the $L^2$-inner product induced by $(J,\theta)$.
We say that $(M^{2n+1},J,\theta)$ is \emph{$R$-singular} if $\ker\Gamma \not= \{ 0 \}$.
As in the case of the scalar curvature of a Riemannian manifold, non-$R$-singular spaces are stable.
More generally:

\begin{theorem}
 \label{thm1}
 Let $(M^{2n+1},J,\theta)$ be a compact, strictly pseudoconvex, partially integrable, non-$R$-singular, pseudohermitian manifold.
 Then $R \colon \mathcal{C} \times \mathcal{K} \to C^\infty(M)$ is a submersion at $(J,\theta)$.
 In particular, there is a neighborhood $U \subset C^\infty(M)$ of $R^{J,\theta}$ such that for any $\psi \in U$, there is a $(\widetilde{J},\widetilde{\theta}) \in \mathcal{C} \times \mathcal{K}$ such that $R^{\widetilde{J},\widetilde{\theta}} = \psi$.
\end{theorem}

See Section~\ref{bg} for an explanation of our notation and Section~\ref{sec:stability} for further discussion.
In Section~\ref{section3} we give examples of $R$-singular and non-$R$-singular spaces.
For example, Proposition~\ref{prop0} states that compact, partially integrable, torsion-free, MTW scalar-flat manifolds are $R$-singular with $\ker\Gamma = \mathbb{R}$, and Proposition~\ref{laplace-kernel} gives a necessary condition on the spectrum of the sublaplacian for a compact, partially integrable, pseudohermitian manifold to be $R$-singular.

Based on the local rigidity results of Fischer--Marsden~\cite{Fischer&Marsden}---or the global rigidity results of Schoen--Yau~\cite{SchoenYau1979} and Gromov--Lawson~\cite{GromovLawson1980,GromovLawson1983}---one wonders whether there is a local rigidity result for the MTW scalar curvature.
We cannot presently prove this, but we do obtain a sufficient condition for infinitesimal rigidity (cf.\ results of Brill--Deser~\cite{BrillDeser1973} and Kazdan--Warner~\cite{KazdanWarner1975} on Riemannian manifolds):

\begin{theorem}
 \label{partial-rigidity}
 Let $(M^{2n+1},J_t,\theta_t)$, $n \geq 2$, be a one-parameter family of compact, integrable, pseudohermitian manifolds with nonnegative MTW scalar curvature.
 Denote by $(E,u) \in T_{(J_0,\theta_0)}(\mathcal{C} \times \mathcal{K})$ the tangent vector to $(J_t,\theta_t)$ at $t=0$.
 Assume that $(J_0,\theta_0)$ is torsion-free and MTW scalar-flat, and that there is a constant
 \begin{equation*}
  C > 1 + \frac{n-1}{n(n+1)(n+2)}
 \end{equation*}
 such that $\int\lvert\nabla_{\bar\gamma}E_{\alpha\beta}\rvert^2 \, \theta \wedge d\theta^n \geq C \int \lvert\nabla^\gamma E_{\gamma\alpha} \rvert^2 \, \theta \wedge d\theta^n$.
 Then $E$ is parallel and $u$ is constant.
\end{theorem}

Note, for example, that Theorem~\ref{partial-rigidity} applies when $\nabla^\gamma E_{\gamma\alpha}=0$.

This paper is organized as follows:

In Section~\ref{bg} we collect the necessary facts about partially integrable pseudohermitian manifolds from Matsumoto's work~\cite{Matsumoto2014} and discuss the space $\mathcal{C} \times \mathcal{K}$ of pairs of partially integrable CR structures and contact forms on a contact manifold.

In Section~\ref{section2} define $R$-singular spaces, compute the linearizations of the MTW scalar curvature and torsion, and characterize $R$-singular spaces in terms of solutions of an explicit PDE.

In Section~\ref{section3} we give some necessary and/or sufficient conditions for a compact, partially integrable, pseudohermitian manifold to be $R$-singular, and discuss some explicit examples.

In Section~\ref{sec:stability} we prove our stability results for the MTW scalar curvature.

In Section~\ref{sec:rigidity} we prove Theorem~\ref{partial-rigidity}.

\section{Notation and terminology}\label{bg}

Let $(M^{2n+1},\xi)$ be a \emph{contact manifold};
i.e.\ $M$ is a smooth $(2n+1)$-dimensional manifold and $\xi \subset TM$ is a distribution of rank $2n$ such that $\theta \wedge d\theta^n \not= 0$ for every local one-form $\theta$ which annihilates $\xi$.
We always assume that $(M^{2n+1},\xi)$ is coorientable;
i.e.\ there is a (globally-defined) one-form $\theta$, called a \emph{contact form}, such that $\ker\theta = \xi$.
Note that $\theta$ is determined up to multiplication by a nowhere-vanishing function.

A \emph{(strictly pseudoconvex) partially integrable CR structure} on $(M^{2n+1},\xi)$ is a vector bundle homomorphism $J \in \End(\xi)$ such that $J^2=\Id$ is the identity, the $(+i)$-eigenspace $T^{1,0} \subset TM \otimes \mathbb{C}$ is partially integrable in the sense that
\begin{equation*}
 [ C^\infty( M ; T^{1,0} ) , C^\infty( M ; T^{1,0} ) ] \subseteq C^\infty( M ; T^{1,0} \oplus T^{0,1} )
\end{equation*}
for $T^{0,1} := \overline{T^{1,0}}$, and the hermitian form $L_\theta(Z,W) := -i \, d\theta(Z,\overline{W}) = i\theta\bigl([Z,\overline{W}]\bigr)$ on $T^{1,0}$ is positive definite for some contact form $\theta$.
We denote by $\mathcal{C}$ and $\mathcal{K}$ the spaces of partially integrable CR structures and contact forms, respectively, on $(M^{2n+1},\xi)$, and always assume that $\mathcal{C}\not=\emptyset$.
We say that $J$ is \emph{integrable} if in fact
\begin{equation*}
 [ C^\infty( M ; T^{1,0} ) , C^\infty( M ; T^{1,0} ) ] \subseteq C^\infty( M ; T^{1,0} ) .
\end{equation*}

A contact form $\theta$ is \emph{positive} if $L_\theta$ is positive definite.
Note that if $\theta$ is positive, then $f\theta$ is positive if and only if $f \in C^\infty(M)$ is positive.

A \emph{(strictly pseudoconvex) pseudohermitian manifold} $(M^{2n+1},J,\theta)$ is a coorientable contact manifold $(M^{2n+1},\xi)$ together with a partially integrable CR structure $J \in \mathcal{C}$ and a positive contact form $\theta$.
The \emph{Reeb vector field} is the unique vector field $T$ such that $\theta(T)=1$ and $d\theta(T,\cdot)=0$.
An \emph{admissible coframe} is a set $\{ \theta^\alpha \}_{\alpha=1}^n$ of locally-defined complex-valued one-forms which annihilate $\mathbb{C}T \oplus T^{0,1}$ and are such that $\{ \theta, \theta^\alpha , \theta^{\bar\alpha} \}$ defines a local frame for $T^\ast M \otimes \mathbb{C}$, where $\theta^{\bar\alpha} := \overline{\theta^\alpha}$.
This data identifies $L_\theta$ with the hermitian matrix $(h_{\alpha\bar\beta})$ defined by
\begin{equation*}
 d\theta = ih_{\alpha\bar\beta} \, \theta^\alpha \wedge \theta^{\bar\beta} .
\end{equation*}
Note that the dual frame $\{ T, Z_\alpha, Z_{\bar\alpha} \}$ to $\{ \theta, \theta^\alpha, \theta^{\bar\alpha} \}$ is such that $\{ Z_\alpha \}$ is a local frame for $T^{1,0}$ and $Z_{\bar\alpha} = \overline{Z_\alpha}$.
We use $h_{\alpha\bar\beta}$ and its inverse $h^{\alpha\bar\beta}$ to lower and raise indices, respectively;
e.g.\ given a section $\tau^\alpha$ of $T^{1,0}$, we denote by $\tau_{\bar\beta} := \tau^\alpha h_{\alpha\bar\beta}$ the corresponding element of the dual space $T^{(0,1)\ast}$.

We parameterize $\mathcal{C}$ as follows:
Pick a background $J_0 \in \mathcal{C}$ and a contact form $\theta$.
Given $J \in \mathcal{C}$, we can locally write
\begin{equation}
 \label{eqn:defn-J}
 J(Z_\alpha) = J_\alpha{}^\beta Z_\beta + J_\alpha{}^{\bar\beta}Z_{\bar\beta} .
\end{equation}
Note that $J_\alpha{}^\beta = i\delta_\alpha^\beta$ and $J_\alpha{}^{\bar\beta}=0$ when $J=J_0$.
The requirement $J^2=-\Id$ yields
\begin{equation*}
 -Z_\alpha = ( J_\alpha{}^\beta J_\beta{}^\gamma + J_\alpha{}^{\bar\beta}J_{\bar\beta}{}^{\gamma} ) Z_\gamma + ( J_\alpha{}^\beta J_\beta{}^{\bar\gamma} + J_\alpha{}^{\bar\beta}J_{\bar\beta}{}^{\bar\gamma} ) Z_{\bar\gamma} ,
\end{equation*}
where $J_{\bar\alpha}{}^{\bar\beta} := \overline{J_\alpha{}^\beta}$ and $J_{\bar\alpha}{}^\beta := \overline{J_\alpha{}^{\bar\beta}}$.
Therefore
\begin{equation}
 \label{eqn:J2-conclusion}
 \begin{aligned}
  J_\alpha{}^\beta J_\beta{}^\gamma + J_\alpha{}^{\bar\beta}J_{\bar\beta}{}^\gamma & = -\delta_\alpha^\gamma , \\
  J_\alpha{}^\beta J_\beta{}^{\bar\gamma} + J_\alpha{}^{\bar\beta}J_{\bar\beta}{}^{\bar\gamma} & = 0 .
 \end{aligned}
\end{equation}
Since $Z_\alpha - iJZ_\alpha$ is in the $(+i)$-eigenspace of $J$, the requirement that $J$ is partially integrable implies that
\begin{equation}
 \label{eqn:partially-integrable-conclusion}
 \begin{aligned}
  0 & = \theta\left( [ Z_\alpha - iJZ_\alpha , Z_\beta - iJZ_\beta ] \right) \\
   & = -J_{\beta\alpha} + J_{\alpha\beta} + iJ_\alpha{}^\gamma J_{\beta\gamma} - iJ_{\alpha\gamma}J_\beta{}^\gamma .
 \end{aligned}
\end{equation}
These observations allow us to identify the formal tangent space $T_{J_0}\mathcal{C}$:

\begin{lem}
 \label{tangent-space}
 Let $(M^{2n+1},\xi)$ be a contact manifold and let $J_0 \in \mathcal{C}$.
 Then
 \begin{equation}
  \label{eqn:tangent-space}
  T_{J_0}\mathcal{C} = \left\{ E_\alpha{}^{\bar\beta} \theta^\alpha \otimes Z_{\bar\beta} + E_{\bar\alpha}{}^\beta \theta^{\bar\alpha} \otimes Z_\beta \mathrel{}:\mathrel{} E_{\bar\alpha}{}^\beta = \overline{E_\alpha{}^{\bar\beta}} , E_{\alpha\beta} = E_{\beta\alpha} \right\} .
 \end{equation}
\end{lem}

\begin{proof}
 Denote by $\mathcal{E}_{J_0} $ the right-hand side of Equation~\eqref{eqn:tangent-space}.

 Let $J_t \in \mathcal{C}$, $t \in (-\varepsilon,\varepsilon)$, be a one-parameter family of partially integrable CR structures with $J_0$ equal to the given background element of $\mathcal{C}$.
 Set $\dot J := \left. \frac{\partial}{\partial t} \right|_{t=0} J$.
 Suppressing the notation $t$, we define the one-parameter families $J_\alpha{}^\beta$, $J_\alpha{}^{\bar\beta}$, and their conjugates by Equation~\eqref{eqn:defn-J}.
 Recall that $J_\alpha{}^\beta(0) = i\delta_\alpha^\beta$ and $J_\alpha{}^{\bar\beta}(0) = 0$.
 Differentiating Equations~\eqref{eqn:J2-conclusion} and evaluating at $t=0$ yields $\dot J_\alpha{}^\beta = 0$.
 Differentiating Equation~\eqref{eqn:partially-integrable-conclusion} and evaluating at $t=0$ yields $\dot J_{\alpha\beta} = \dot J_{\beta\alpha}$.
 Therefore $T_{J_0}\mathcal{C} \subseteq \mathcal{E}_{J_0}$.
 
 Now let $E \in \mathcal{E}_{J_0} $.
 Then $J_0E + EJ_0 = 0$, and hence
 \begin{equation*}
  J_0\exp(-J_0E) = \exp(J_0E)J_0 .
 \end{equation*}
 Thus $J_0\exp(-J_0E) \in \mathcal{C}$.
 It follows that $J_t := J_0\exp(-tJ_0E)$ is a path in $\mathcal{C}$ with $\dot J = E$.
 Hence $\mathcal{E}_{J_0} \subseteq T_{J_0}\mathcal{C}$.
\end{proof}

Matsumoto showed~\cite{Matsumoto2014}*{Proposition~3.1} that there is a unique connection on $(M^{2n+1},J,\theta)$ for which $\ker\theta$, $T$, $J$, and $L_\theta$ are parallel and the torsion is as simple as possible, generalizing the Tanaka--Webster connection~\cites{Webster,Tanaka}.
More precisely, given an admissible coframe $\{ \theta^\alpha \}$, define the \emph{connection one-forms} $\omega_\alpha{}^\beta$ by
\begin{align*}
 d\theta^\alpha & = \theta^\beta \wedge \omega_\beta{}^\alpha + \theta \wedge \tau^\alpha - \frac{1}{2}N_{\bar\beta\bar\sigma}{}^\alpha \theta^{\bar\beta} \wedge \theta^{\bar\sigma} , & \tau^\alpha & = A^\alpha{}_{\bar\beta}\theta^{\bar\beta} , \\
 dh_{\alpha\bar\beta} & = \omega_{\alpha\bar\beta} + \omega_{\bar\beta\alpha} , & \omega_{\bar\beta\alpha} & = \overline{\omega_{\beta\bar\alpha}},
\end{align*}
where $N_{\bar\alpha\bar\beta}{}^\gamma$ denotes the restriction
\begin{equation*}
 N(\bar Z,\bar W) := \pi_{T^{1,0}}[\bar Z,\bar W]
\end{equation*}
of the Nijenhuis tensor to $T^{0,1} \otimes T^{0,1}$.
Note that $N_{\bar\alpha\bar\beta}{}^\gamma=0$ if and only if $J$ is integrable.
Matsumoto also showed~\cite{Matsumoto2014}*{Proposition~3.4} that $A_{\alpha\beta} = A_{\beta\alpha}$.
The \emph{MTW connection} is $\nabla Z_\alpha := \omega_\alpha{}^\beta \otimes Z_\beta$.
Matsumoto showed~\cite{Matsumoto2014}*{Lemma~3.5} that
\begin{equation}
 \label{eqn:matsumoto-commutator}
 \begin{split}
  \nabla_\beta\nabla_\alpha u & = \nabla_\alpha\nabla_\beta u - N_{\alpha\beta}{}^{\bar\gamma}\nabla_{\bar\gamma} , \\
  \nabla_{\bar\beta}\nabla_\alpha u & = \nabla_\alpha\nabla_{\bar\beta}u + ih_{\alpha\bar\beta}\nabla_0u , \\
  \nabla_\alpha\nabla_0 u & = \nabla_0 \nabla_\alpha u + A_\alpha{}^{\bar\beta}\nabla_{\bar\beta} u ,
 \end{split}
\end{equation}
for all $u \in C^\infty(M)$.

The \emph{MTW scalar curvature} $R := R_\alpha{}^\alpha{}_\beta{}^\beta$ is determined via the curvature two-forms $\Omega_\alpha{}^\beta := d\omega_\alpha{}^\beta - \omega_\alpha{}^\gamma \wedge \omega_\gamma{}^\beta$:
\begin{equation}
 \label{eqn:curvature-equations}
 \begin{aligned}
  \Omega_\alpha{}^\beta & \equiv R_\alpha{}^\beta{}_{\gamma\bar\epsilon} \, \theta^\gamma \wedge \theta^{\bar\epsilon} + V_{\alpha}{}^\beta{}_{\gamma\epsilon}\theta^\gamma \wedge \theta^\epsilon + V_\alpha{}^\beta{}_{\bar\gamma\bar\epsilon}\theta^{\bar\gamma} \wedge\theta^{\bar\epsilon} \mod \theta , \\
  V_{\alpha}{}^\beta{}_{\gamma\epsilon} & = i\delta_{[\gamma}^\beta A_{\epsilon]\alpha} + \frac{1}{2}\nabla^\beta N_{\gamma\epsilon\alpha} , \\
  V_{\alpha}{}^\beta{}_{\bar\gamma\bar\epsilon} & = ih_{\alpha[\bar\epsilon}A_{\bar\gamma]}{}^\beta - \frac{1}{2}\nabla_\alpha N_{\bar\epsilon\bar\gamma}{}^\beta ;
 \end{aligned}
\end{equation}
see~\cite{Matsumoto2014}*{Equations~(3.7)--(3.9)}.
This generalizes the Tanaka--Webster scalar curvature to partially integrable pseudohermitian manifolds.

We also have commutator identities (cf.\ \cite{Lee1}*{Lemma~2.3}):
If $\sigma = \sigma_\alpha\theta^\alpha$, then
\begin{equation*}
 d\sigma_\alpha = \nabla_\beta\sigma_\alpha \, \theta^\beta + \nabla_{\bar\gamma}\sigma_\alpha \, \theta^{\bar\gamma} + \nabla_0\sigma_\alpha \, \theta + \sigma_\varepsilon \omega_\alpha{}^\varepsilon .
\end{equation*}
Computing the $\theta^\beta\wedge\theta^{\bar\gamma}$- and $\theta \wedge \theta^{\bar\beta}$-components of $d^2\sigma_\alpha = 0$ yields
\begin{align}
 \label{eqn:curv-commutator} \nabla_{\bar\gamma}\nabla_\beta\sigma_{\alpha} & = \nabla_\beta\nabla_{\bar\gamma}\sigma_{\alpha} + ih_{\beta\bar\gamma}\nabla_0\sigma_{\alpha} + R_\alpha{}^\varepsilon{}_{\beta\bar\gamma}\sigma_{\varepsilon} , \\
 \label{eqn:tor-commutator} \nabla_{\bar\beta}\nabla_0\sigma_{\alpha} & = \nabla_0\nabla_{\bar\beta}\sigma_{\alpha} + A_{\bar\beta}{}^\varepsilon\nabla_\varepsilon\sigma_{\alpha} + \sigma_\varepsilon\nabla_\alpha A_{\bar\beta}{}^\varepsilon - A_\alpha{}^{\bar\rho}N_{\bar\beta\bar\rho}{}^\epsilon\sigma_\epsilon ,
\end{align}
respectively.

\section{$R$-singular spaces}\label{section2}

The MTW scalar curvature of a pseudohermitian manifold depends on both the choice of CR structure and the choice of contact form.
Motivated by general terminology in the Riemannian setting~\cite[Definition~5.5]{CaseLinYuan}, we say that a pseudohermitian manifold $(M,J,\theta)$ is $R$-singular if the adjoint of the linearization of the MTW scalar curvature has nontrivial kernel.
This section will make that notion precise.

First we use general transformation formulas of Matsumoto~\cite{Matsumoto2014} to compute the linearization of the MTW scalar curvature and torsion over all choices of contact form.
These generalize formulas of Lee~\cite{Lee1}.

\begin{lem}
 \label{cr-linearization}
 Let $(M^{2n+1},J,\theta)$ be a pseudohermitian manifold and let $\theta_t \in \mathcal{K}$, $t \in (-\varepsilon,\varepsilon)$, be a one-parameter family of contact forms with $\theta_0=\theta$.
 Define $u \in C^\infty(M)$ by $u\theta = \left. \frac{\partial}{\partial t}\right|_{t=0} \theta_t$.
 Then
 \begin{align}
  \label{eqn:cr-linearize-R} \left. \frac{\partial}{\partial t} \right|_{t=0} R^{J,\theta_t} & = -(n+1)\Delta_bu - Ru , \\
  \label{eqn:cr-linearize-A} \left. \frac{\partial}{\partial t} \right|_{t=0} A_{\alpha\beta}^{J,\theta_t} & = \frac{i}{2}(u_{\alpha\beta} + u_{\beta\alpha}) + \frac{i}{2}(N_{\gamma\alpha\beta} + N_{\gamma\beta\alpha})u^\gamma ,
 \end{align}
 where $\Delta_b := 2\Real\nabla^\alpha\nabla_\alpha$.
\end{lem}

\begin{proof}
 Let $\widehat{\theta} = e^{\Upsilon}\theta$.
 Taking the trace of Matsumoto's formula~\cite{Matsumoto2014}*{Equation~(3.13)} for the Ricci tensors $R_{\alpha\bar\beta} := R_\gamma{}^\gamma{}_{\alpha\bar\beta}$ and $\widehat{R}_{\alpha\bar\beta}$ of $\theta$ and $\widehat{\theta}$, respectively, yields
 \begin{equation*}
  e^{\Upsilon}\widehat{R} = R - (n+1)\Delta_b\Upsilon - n(n+1)\Upsilon_\gamma \Upsilon^\gamma .
 \end{equation*}
 Linearizing yields Equation~\eqref{eqn:cr-linearize-R}.
 Equation~\eqref{eqn:cr-linearize-A} follows by linearizing Matsumoto's formula~\cite{Matsumoto2014}*{Equation~(3.12)} for the torsion of $\widehat{\theta}$.
\end{proof}

Next we derive the linearizations of the MTW scalar curvature and torsion over all choices of partially integrable CR structures.
These generalize formulas of Cheng and Lee~\cite{ChengLee} and Afeltra, Cheng, Malchiodi, and Yang~\cite{Afeltra} in three and general dimensions, respectively.

\begin{prop}
 \label{prop1.3}
 Let $(M^{2n+1},J,\theta)$ be a pseudohermitian manifold and let $J_t \in \mathcal{C}$, $t \in (-\varepsilon,\varepsilon)$, be a one-parameter family of CR structures with $J_0 = J$.
 Define $\dot{J} \in \End(\ker\theta)$ by $\dot{J} := \left.\frac{\partial}{\partial t}\right|_{t=0} J_t$.
 Then
 \begin{align}
  \label{eqn:linearize-R-CR} \left. \frac{\partial}{\partial t} \right|_{t=0} R^{J_t,\theta} & = 2\Real (i\nabla^\alpha\nabla^\beta E_{\alpha\beta} + i\nabla_\alpha(N^{\beta\alpha\gamma}E_{\beta\gamma}) - nA^{\alpha\beta}E_{\alpha\beta}) , \\
  \label{eqn:linearize-torsion-CR} \left. \frac{\partial}{\partial t} \right|_{t=0} A_{\alpha\beta}^{J_t,\theta} & = i\nabla_0E_{\alpha\beta} ,
 \end{align}
 where $\dot{J} = E_{\alpha}{}^{\bar\beta}\theta^\alpha \otimes Z_{\bar\beta} + E_{\bar\alpha}{}^{\beta}\theta^{\bar\alpha} \otimes Z_{\beta}$.
\end{prop}

\begin{proof}
 Lemma~\ref{tangent-space} implies that we may write $\dot J = 2\Real E_\alpha{}^{\bar\beta}\theta^\alpha \otimes Z_{\bar\beta}$.

 Pick $p \in M$ and let $\{ \theta^\alpha \}$ be an admissible coframe for $(M^{2n+1},J,\theta)$ such that $\omega_\alpha{}^\beta(p)=0$.
 Set
 \begin{align*}
  \theta_t^\alpha & := \theta^\alpha - iJ_t\theta^\alpha = ( \delta_\beta^\alpha - iJ_\beta{}^\alpha )\theta^\beta - iJ_{\bar\beta}{}^\alpha\theta^{\bar\beta} ,
 \end{align*}
 where $J_\alpha{}^\beta$ and $J_\alpha{}^{\bar\beta}$ denote the components of $J_t$ as in Equation~\eqref{eqn:defn-J}.
 Then $\{ \theta_t^\alpha \}$ is an admissible coframe for $(M^{2n+1},J_t,\theta)$.
 Clearly $\dot\theta^\alpha = -iE_{\bar\beta}{}^\alpha\theta^{\bar\beta}$.
 Since $\theta$ is fixed, we deduce that
 \begin{equation*}
  0 = d\dot\theta = i\dot h_{\alpha\bar\beta} \theta^\alpha \wedge \theta^{\bar\beta} .
 \end{equation*}
 Hence $\dot h_{\alpha\bar\beta} = 0$.

 We first compute the linearization of the connection one-forms.
 In this paragraph we compute modulo $\theta^{\bar\beta} \wedge \theta^{\bar\gamma}$.
 On the one hand, differentiating the equation for $d\theta^\alpha$ at $t=0$ yields
 \begin{align*}
  d\dot\theta^\alpha & \equiv \theta^\beta \wedge \dot\omega_\beta{}^\alpha + \dot A^\alpha{}_{\bar\gamma} \theta \wedge \theta^{\bar\gamma} + A^\alpha{}_{\bar\gamma} \theta \wedge \dot\theta^{\bar\gamma} - \frac{1}{2}N_{\bar\gamma\bar\sigma}{}^\alpha (\dot\theta^{\bar\gamma} \wedge \theta^{\bar\sigma} + \theta^{\bar\gamma} \wedge \dot\theta^{\bar\sigma}) \\
  & \equiv \theta^\beta \wedge \left( \dot\omega_\beta{}^\alpha - iE_\beta{}^{\bar\gamma}A_{\bar\gamma}{}^\alpha \theta - iE_\beta{}^{\bar\gamma}N_{\bar\gamma\bar\sigma}{}^\alpha\theta^{\bar\sigma} \right) + \dot A^\alpha{}_{\bar\gamma}\theta \wedge \theta^{\bar\gamma} ,
 \end{align*}
 where the second equality uses the skew symmetry $N_{[\bar\gamma\bar\sigma]}{}^\alpha = N_{\bar\gamma\bar\sigma}{}^\alpha$ of the Nijenhuis tensor.
 On the other hand, our formula for $\dot\theta^\alpha$ yields
 \begin{align*}
  d\dot\theta^\alpha & = -id(E_{\bar\gamma}{}^\alpha\theta^{\bar\gamma}) \\
  & \equiv -i\theta^\beta \wedge \left( \nabla_\beta E_{\bar\gamma}{}^\alpha \theta^{\bar\gamma} - A_\beta{}^{\bar\gamma}E_{\bar\gamma}{}^\alpha\theta - \frac{1}{2}N_{\beta\sigma}{}^{\bar\gamma}E_{\bar\gamma}{}^\alpha \theta^\sigma \right) - i\nabla_0E_{\bar\beta}{}^\alpha \theta \wedge \theta^{\bar\beta} \\
  & \equiv -i\theta^\beta \wedge \left( \nabla_\beta E_{\bar\gamma}{}^\alpha \theta^{\bar\gamma} - A_\beta{}^{\bar\gamma}E_{\bar\gamma}{}^\alpha\theta - E^\alpha{}_{\bar\gamma} N^{\bar\gamma}{}_{\sigma\beta}  \theta^{\sigma} \right) - i\nabla_0E_{\bar\beta}{}^\alpha \theta \wedge \theta^{\bar\beta} ,
 \end{align*}
 at $p$, where the last equality uses the identity $N_{[\alpha\beta\gamma]}=0$.
 Comparing the $\theta \wedge \theta^{\bar\beta}$ coefficients yields Equation~\eqref{eqn:linearize-torsion-CR}.
 Comparing the $\theta^\beta$ coefficients and using the equation $\dot\omega_{\alpha\bar\beta} + \dot\omega_{\bar\beta\alpha} = d\dot h_{\alpha\bar\beta} = 0$ yields
 \begin{multline*}
  \dot\omega_\beta{}^\alpha = -i\Bigl( (\nabla^\alpha E_{\gamma\beta} - E^\alpha{}_{\bar\sigma} N^{\bar\sigma}{}_{\gamma\beta})\theta^\gamma + \left( \nabla_\beta E_{\bar\gamma}{}^\alpha - E_\beta{}^{\bar\sigma}N_{\bar\sigma\bar\gamma}{}^\alpha \right)\theta^{\bar\gamma} \\
   - \left( E_\beta{}^{\bar\gamma} A_{\bar\gamma}{}^\alpha + A_\beta{}^{\bar\gamma}E_{\bar\gamma}{}^\alpha \right) \theta \Bigr) .
 \end{multline*}

 We now compute the linearization of the MTW scalar curvature.
 On the one hand, differentiating the equation $\Omega_\alpha{}^\beta = d\omega_\alpha{}^\beta - \omega_\alpha{}^\gamma \wedge \omega_\gamma{}^\beta$ at $p$ yields
 \begin{align*}
  \dot\Omega_\alpha{}^\beta & = d\dot\omega_\alpha{}^\beta \\
   & \equiv \Bigl( i\nabla_{\bar\epsilon}\nabla^\beta E_{\gamma\alpha} - i\nabla_\gamma \nabla_\alpha E_{\bar\epsilon}{}^\beta - i\nabla_{\bar\epsilon}(E^\beta{}_{\bar\zeta}N^{\bar\zeta}{}_{\gamma\alpha}) + i\nabla_\gamma(E_\alpha{}^{\bar\zeta}N_{\bar\zeta\bar\epsilon}{}^\beta) \\
    & \qquad  - (E_\alpha{}^{\bar\zeta}A_{\bar\zeta}{}^\beta + A_{\alpha}{}^{\bar\zeta}E_{\bar\zeta}{}^\beta)h_{\gamma\bar\epsilon} \Bigr) \theta^{\gamma} \wedge \theta^{\bar\epsilon} \mod \theta, \theta^{\gamma} \wedge \theta^{\epsilon} , \theta^{\bar\gamma} \wedge \theta^{\bar\epsilon} .
 \end{align*}
 On the other hand, the definition~\eqref{eqn:curvature-equations} of $R_{\alpha\bar\beta\gamma\bar\epsilon}$ yields
 \begin{align*}
  \dot\Omega_\alpha{}^\beta & \equiv \left( \dot R_\alpha{}^\beta{}_{\gamma\bar\epsilon} - 2iV_\alpha{}^\beta{}_{\gamma\zeta}E^{\zeta}{}_{\bar\epsilon} - 2iV_\alpha{}^\beta{}_{\bar\epsilon\bar\zeta}E^{\bar\zeta}{}_{\gamma} \right) \theta^\gamma \wedge \theta^{\bar\epsilon} \mod \theta, \theta^\gamma \wedge \theta^{\epsilon}, \theta^{\bar\gamma} \wedge \theta^{\bar\epsilon} .
 \end{align*}
 Therefore
 \begin{align*}
  \dot R_\alpha{}^\beta{}_{\gamma\bar\epsilon} & = i\nabla_{\bar\epsilon}\nabla^\alpha E_{\gamma\beta} - i\nabla_\gamma\nabla_\beta E_{\bar\epsilon}{}^\alpha - i\nabla_{\bar\sigma}(E^\alpha{}_{\bar\zeta}N^{\bar\zeta}{}_{\gamma\beta}) + i\nabla_\gamma(E_\alpha{}^{\bar\zeta}N_{\bar\zeta\bar\epsilon}{}^\beta) \\
   & \quad - (E_\alpha{}^{\bar\zeta}A_{\bar\zeta}{}^\beta + A_\alpha{}^{\bar\zeta}E_{\bar\zeta}{}^\beta)h_{\gamma\bar\epsilon} + 2iV_\alpha{}^\beta{}_{\gamma\zeta}E^\zeta{}_{\bar\epsilon} + 2iV_\alpha{}^\beta{}_{\bar\zeta\bar\epsilon}E^{\bar\zeta}{}_\gamma .
 \end{align*}
 Contracting using the identities $\dot h_{\alpha\bar\beta}=0$ and $V_{\alpha\bar\beta(\gamma\epsilon)}=0$ yields Equation~\eqref{eqn:linearize-R-CR}.
\end{proof}

Recall that $\mathcal{K}$ is the space of positive contact forms on $(M^{2n+1},\xi)$.
As in Lemma~\ref{cr-linearization}, we identify $T_\theta\mathcal{K} \cong C^\infty(M)$ by
\begin{equation*}
 \left. \frac{\partial}{\partial t} \right|_{t=0} e^{\Upsilon_t}\theta \cong \left. \frac{\partial}{\partial t} \right|_{t=0} \Upsilon_t
\end{equation*}
for $\Upsilon_t$ a one-parameter family of functions with $\Upsilon_0=0$.
This allows us to compute the linearization $DR \colon T_{(J,\theta)}(\mathcal{C} \times \mathcal{K}) \to C^\infty(M)$,
\begin{equation*}
 DR(E,u) := \left. \frac{\partial}{\partial t} \right|_{t=0} R^{J+tE,e^{tu}\theta} ,
\end{equation*}
of the MTW scalar curvature:
\begin{multline}
 \label{eqn:linearization}
 DR(E,u) = 2\Real\left( i\nabla^\alpha\nabla^\beta E_{\alpha\beta} + i\nabla_\alpha(N^{\beta\alpha\gamma}E_{\beta\gamma}) - nA^{\alpha\beta}E_{\alpha\beta} \right) \\
  - (n+1)\Delta_bu - Ru .
\end{multline}

Let $\Gamma \colon C^\infty(M) \to T_{(J,\theta)}(\mathcal{C} \times \mathcal{K})$ denote the adjoint of $DR$ with respect to the $L^2$-metric on $T_{(J,\theta)}(\mathcal{C} \times \mathcal{K})$ induced by $(J,\theta)$.

\begin{definition}
 A pseudohermitian manifold $(M^{2n+1},J,\theta)$ is \emph{$R$-singular} if there is a nonzero $f \in C^\infty(M)$ such that $\Gamma f = 0$.
\end{definition}

We explicitly characterize $R$-singular spaces:

\begin{lem}
 \label{characterize-R-signular}
 A pseudohermitian manifold $(M^{2n+1},J,\theta)$ is \emph{$R$-singular} if and only if there is an $f \in C^\infty(M)$ such that
 \begin{align}
  \label{eqn:characterize-torsion} 0 & = i\nabla_\alpha\nabla_\beta f - iN_{\alpha\gamma\beta}\nabla^\gamma f + nfA_{\alpha\beta} , \\
  \label{eqn:characterize-scalar} 0 & = (n+1)\Delta_bf + Rf .
 \end{align}
\end{lem}

\begin{proof}
 Let $(E,u) \in T_{(J,\theta)}(\mathcal{C} \times \mathcal{K})$ be compactly-supported.
 We compute that
 \begin{equation}
  \label{eqn:adjoint-linearization}
  \begin{split}
  \int_M \langle \Gamma f, (E,u) \rangle \, \theta \wedge d\theta^n & = \int_M fDR(E,u) \, \theta \wedge d\theta^n \\
  & = \int_M \Bigl[ 2\Real(i\nabla^\alpha\nabla^\beta f - iN^{\alpha\gamma\beta}\nabla_\gamma f - nfA^{\alpha\beta})E_{\alpha\beta} \\
   & \qquad\qquad - \bigl((n+1)\Delta_bf + Rf\bigr)u \Bigr] \, \theta \wedge d\theta^n .
  \end{split}
 \end{equation}
 The conclusion readily follows.
\end{proof}

\section{Examples of $R$-singular and non-$R$-singular spaces}\label{section3}

Here we give some examples of $R$-singular and non $R$-singular spaces.
First, compact, torsion-free, MTW scalar-flat manifolds are $R$-singular.
Indeed:

\begin{prop}\label{prop0}
 Let $(M^n,J,\theta)$ be a compact pseudohermitian manifold.
 The following are equivalent:
 \begin{enumerate}
  \item $\mathbb{R} \subseteq \ker\Gamma$.
  \item $\mathbb{R} = \ker\Gamma$.
  \item $R=0$ and $A_{\alpha\beta}=0$.
 \end{enumerate}
\end{prop}

\begin{proof}
 It suffices to show that $(1) \Rightarrow (3)$ and $(3) \Rightarrow (2)$.

 Suppose that $1 \in \ker\Gamma$.
 Applying Lemma~\ref{characterize-R-signular} yields $R=0$ and $A_{\alpha\beta}=0$.

 Suppose now that $R=0$ and $A_{\alpha\beta}=0$.
 On the one hand, Equation~\eqref{eqn:characterize-scalar} implies that $\ker\Gamma \subseteq \mathbb{R}$.
 On the other hand, Equations~\eqref{eqn:characterize-torsion} and~\eqref{eqn:characterize-scalar} imply that $\mathbb{R} \subseteq \ker\Gamma$.
\end{proof}

A slight modification to this argument shows that compact scalar-flat pseudohermitian manifolds are $R$-singular if and only if they are torsion-free.

\begin{prop}\label{prop7}
 Let $(M^{2n+1},J,\theta)$ be a compact pseudohermitian manifold with vanishing MTW scalar curvature.
 If $A_{\alpha\beta}\not=0$, then $\ker\Gamma = \{ 0 \}$.
\end{prop}

\begin{proof}
 Let $f \in \ker\Gamma$ be nonzero.
 Equation~\eqref{eqn:characterize-scalar} implies that $f$ is constant.
 We conclude from Equation~\eqref{eqn:characterize-torsion} that $A_{\alpha\beta}=0$.
\end{proof}

The above observations all exploit the fact that the kernel of the sublaplacian equals the constants.
One can obtain a more general necessary condition for a space to be $R$-singular by considering the spectrum of a different operator.

\begin{prop}
 \label{laplace-kernel}
 Let $(M^{2n+1},J,\theta)$ be a compact pseudohermitian manifold.
 If $\ker\bigl(-\Delta_b - \frac{R}{n+1}\bigr) = \{ 0 \}$, then $\ker\Gamma = 0$.
\end{prop}

\begin{remark}
 The condition $\ker\bigl(-\Delta_b - \frac{R}{n+1}\bigr)=\{0\}$ is sometimes easy to check.
 For example, it is always true if
 \begin{enumerate}
  \item $R<0$; or
  \item $R$ is constant and $\frac{R}{n+1}$ is not in the spectrum of $-\Delta_b$.
 \end{enumerate}
\end{remark}

\begin{proof}
 This follows immediately from Equation~\eqref{eqn:characterize-scalar}.
\end{proof}

Next we consider Sasaki manifolds $(M^{2n+1},J,\theta)$;
i.e.\ compact \emph{integrable} pseudohermitian manifolds equipped with a torsion-free contact form.
The basic cohomology groups---obtained by restricting the exterior derivative to forms $\alpha$ for which both $i_T\alpha=0$ and $i_Td\alpha=0$---are important topological invariants of these spaces~\cite{BoyerGalicki2008}.
Here we give a sufficient condition for there to be no nontrivial basic functions in $\ker\Gamma$.

\begin{prop}
 \label{sasakian}
 Let $(M^{2n+1},J,\theta)$ be a compact, integrable, pseudohermitian manifold with vanishing torsion.
 Suppose that $R_{\alpha\bar\beta} = \frac{R}{n}h_{\alpha\bar\beta}$ and $R>0$.
 If $f \in \ker\Gamma$ is such that $Tf=0$, then $f=0$.
\end{prop}

\begin{proof}
 Since $\Gamma f = 0$ and both $A_{\alpha\beta}=0$ and $N_{\alpha\beta\gamma}=0$, Equation~\eqref{eqn:characterize-torsion} implies that $f_{\alpha\beta}=0$.
 Since $T=0$, we deduce from Lee's commutator identities~\cite{Lee1}*{Lemma~2.3} that
 \begin{align*}
  \Delta_bf & = 2f_\beta{}^\beta , & f_{\beta}{}^\beta{}_\alpha & = -R_\alpha{}^\beta f_\beta = -\frac{R}{n}f_\alpha .
 \end{align*}
 Next, Equation~\eqref{eqn:characterize-scalar} implies that $Rf = -(n+1)\Delta_bf$.
 Differentiating this yields
 \begin{align*}
  0 & = f\nabla_\alpha R + Rf_\alpha + 2(n+1)f_\beta{}^\beta{}_\alpha \\
  & = f\nabla_\alpha R - \frac{n+2}{n}Rf_\alpha .
 \end{align*}
 Since $R$ is positive, we deduce that there is a $c \in \mathbb{R}$ such that $f^{n+2}=cR^n$.
 Inserting this into the equation $\int Rf = -(n+1)\int \Delta_bf = 0$ yields $c=0$.
\end{proof}

We now turn to two explicit examples.
First, the standard pseudohermitian sphere is an $R$-singular space.

\begin{example}
 \label{ex:sphere}
 Let $u \in C^\infty(\mathbb{C}^{n+1})$ be the function $u(z) = \lvert z \rvert^2-1$ and denote $S^{2n+1} := u^{-1}(\{0\})$.
 Set $\theta := \Imaginary \partial u$.
 Note that $\ker\theta = TS^{2n+1} \cap J_0(TS^{2n+1})$, where $J_0$ is the usual complex structure on $\mathbb{C}^{n+1}$.
 Set $J := J_0\rvert_{\ker\theta}$.
 Then $(S^{2n+1},J,\theta)$ is the \emph{standard pseudohermitian $(2n+1)$-sphere};
 it is an integrable pseudohermitian manifold with vanishing torsion and $R_{\alpha\bar\beta\gamma\bar\sigma} = h_{\alpha\bar\beta}h_{\gamma\bar\sigma} + h_{\alpha\bar\sigma}h_{\gamma\bar\beta}$, where $h_{\alpha\bar\beta}$ denotes the Levi form.
 In particular, $R = n(n+1)$ is constant.
 By considering the decomposition of $C^\infty(S^{2n+1})$ into spherical harmonics, we conclude (cf.\ \citelist{ \cite{Cheng1991}*{Section~2} \cite{Stanton1989}*{Section~4} }) that $\ker\Gamma$ is $(2n+2)$-dimensional and spanned by the coordinate functions in $\mathbb{R}^{2n+2} \cong \mathbb{C}^{n+1}$.
\end{example}

Second, Rossi spheres are not $R$-singular spaces.

\begin{example}
 \label{ex:rossi}
 Let $(S^3,J,\theta)$ be the pseudohermitian three-sphere.
 Given $t \in \mathbb{R}$ with $0 < \lvert t \rvert < 1$, the \emph{Rossi sphere} $(S^3,J_t,\theta)$ is the pseudohermitian manifold with $J_t$ the (integrable) CR structure on $\ker\theta$ whose $(+i)$-eigenspace is spanned by $Z_1^t := \frac{1}{\sqrt{1-t^2}}(Z_1 + tZ_{\bar 1})$, where $Z_1 := \bar z^2 \partial_{z^1} - \bar z^1 \partial_{z_2}$.
 Note that $J_0$ is the standard CR structure on $S^3$.

 Fix $t$ such that $0 < \lvert t \rvert < 1$ and suppose that $f \in \ker\Gamma$.
 Set $\theta_0^1 := z^2\,dz^1 - z^1\,dz^2$ and $\theta_t^1 := \frac{1}{\sqrt{1-t^2}}(\theta_0^1 - t\theta_0^{\bar1})$.
 Chanillo, Chiu, and Yang showed~\cite{Chanillo&Chiu&Yang}*{Proposition~2.5} that $\{ \theta_t^1 \}$ is an admissible coframe for $(S^3,J_t,\theta)$ with dual frame $\{ Z_1^t \}$ and whose Levi form, connection one-forms, torsion, and Tanaka--Webster scalar curvature are given by
 \begin{align*}
  h_{1\bar1} & = 1 , & \omega_1{}^1 & = -\frac{2(1+t^2)}{1-t^2}i\theta , \\
  A_{11} & = -\frac{4it}{1-t^2} ,& R & = \frac{2(1+t^2)}{1-t^2} ,
 \end{align*}
 respectively.
 Let $f \in \ker \Gamma$.
 On the one hand, Equation~\eqref{eqn:characterize-torsion} yields
 \begin{align*}
  0 & = -i(1-t^2)(iZ_1^tZ_1^tf + A_{11}f) \\
   & = Z_1Z_1 f + t(Z_1Z_{\bar 1} + Z_{\bar1}Z_1)f + t^2Z_{\bar1}Z_{\bar1}f - 4tf .
 \end{align*}
 Taking twice the real part gives
 \begin{equation*}
  0 = (1+t^2)(Z_1Z_1 + Z_{\bar1}Z_{\bar1})f + 2t(Z_1Z_{\bar1} + Z_{\bar1}Z_1)f - 8tf .
 \end{equation*}
 On the other hand, Equation~\eqref{eqn:characterize-scalar} yields
 \begin{align*}
  0 & = \frac{1-t^2}{2}\bigl( 2(Z_1^tZ_{\bar1}^t + Z_{\bar1}^tZ_1) + Rf \bigr) \\
   & = 2t(Z_1Z_1 +Z_{\bar1}Z_{\bar1})f + (1+t^2)(Z_1Z_{\bar1}+Z_{\bar1}Z_1)f + (1+t^2)f .
 \end{align*}
 Combining the previous two displays yields
 \begin{align}
  \label{eqn:rossi-scalar} -(1-t^2)^2\Delta_bf & = (1 + 18t^2 + t^4)f , \\
  \label{eqn:rossi-torsion} (1-t^2)^2(Z_1Z_1 + Z_{\bar1}Z_{\bar1})f & = 10t(1+t^2)f ,
 \end{align}
 where $\Delta_b = Z_1Z_{\bar1} + Z_{\bar1}Z_1$ is the sublaplacian on the pseudohermitian three-sphere.

 Now let $\mathcal{H}^{p,q}$ denote the set of restrictions to $S^{3}$ of homogeneous harmonic polynomials on $\mathbb{C}^2$ which are linear combinations of $z^\alpha \bar z^\beta$, where $\alpha,\beta \in \mathbb{N}_0^2$ are multi-indices with $\lvert\alpha\rvert=p$ and $\lvert\beta\rvert = q$.
 Note that $Z_1Z_1(\mathcal{H}^{p,q}) \subseteq \mathcal{H}^{p-2,q+2}$, and the image is nontrivial if $p \geq 2$.
 Stanton showed~\cite{Stanton1989}*{Proposition~4.4} that if $f \in \mathcal{H}^{p,q}$, then $\Delta_bf = -(2pq+p+q)f$.
 Equation~\eqref{eqn:rossi-scalar} implies that $f \in \bigoplus_{p+q=k}\mathcal{H}^{p,q}$ with
 \begin{equation}
  \label{eqn:eigenvalue}
  2pq + p + q = \frac{1+18t^2+t^4}{(1-t^2)^2} ,
 \end{equation}
 while Equation~\eqref{eqn:rossi-torsion} implies that if the projection of $f$ to $\mathcal{H}^{p,q}$ is nontrivial, then its projection to $\mathcal{H}^{p+2j,q-2j}$ is nontrivial for each $j \in \mathbb{Z}$ such that $-p \leq 2j \leq q$.
 It follows that if $f\not=0$, then $f \in \Real(\mathcal{H}^{2,0} \oplus \mathcal{H}^{0,2})$ or $f \in \Real(\mathcal{H}^{3,1} \oplus \mathcal{H}^{1,3})$.
 The conclusion follows by ruling out each of these cases:

 On the one hand, $\Real(\mathcal{H}^{2,0} \oplus \mathcal{H}^{0,2})$ is spanned by
 \begin{equation*}
  \{ \Real (z^1)^2, \Real (z^2)^2, \Imaginary (z^1)^2, \Imaginary (z^2)^2, \Real(z^1z^2), \Imaginary(z^1z^2) \} .
 \end{equation*}
 Direct computation shows that the eigenvalues of $Z_1Z_1 + Z_{\bar1}Z_{\bar1}$ on $\Real(\mathcal{H}^{2,0} \oplus \mathcal{H}^{0,2})$ are $\pm2$.
 Inserting this and Equation~\eqref{eqn:eigenvalue} into Equations~\eqref{eqn:rossi-scalar} and~\eqref{eqn:rossi-torsion} yields
 \begin{align*}
  2(1-t^2)^2 & = 1 + 18t^2 + t^4 , & \pm 2(1-t^2)^2 = 10t(1+t^2) .
 \end{align*}
 This system has no solution.

 On the other hand, $\Real(\mathcal{H}^{3,1} \oplus \mathcal{H}^{1,3})$ is spanned by
 \begin{align*}
  & \Bigl\{ \Real (z^1)^3\bar z^2 , \Real (z^2)^3\bar z^1, \Imaginary (z^1)^3\bar z^2, \Imaginary (z^2)^3\bar z^1, \Real\bigl( (z^1)^2(\lvert z^1\rvert^2 - 3\lvert z^2\rvert^2)\bigr), \\
  & \qquad \Real\bigl( (z^2)^2(\lvert z^2\rvert^2 - 3\lvert z^1\rvert^2)\bigr), \Imaginary\bigl( (z^1)^2(\lvert z^1\rvert^2 - 3\lvert z^2\rvert^2)\bigr), \Imaginary\bigl( (z^2)^2(\lvert z^2\rvert^2 - 3\lvert z^1\rvert^2)\bigr), \\
  & \qquad  \Real\bigl( z^1z^2(\lvert z^1\rvert^2 - \lvert z^2\rvert^2) \bigr), \Imaginary\bigl( z^1z^2(\lvert z^1\rvert^2 - \lvert z^2\rvert^2) \bigr) \Bigr\} .
 \end{align*}
 Direct computation shows that the eigenvalues of $Z_1Z_1 + Z_{\bar1}Z_{\bar1}$ on $\Real(\mathcal{H}^{3,1} \oplus \mathcal{H}^{1,3})$ are $\pm6$.
 Inserting this and Equation~\eqref{eqn:eigenvalue} into Equations~\eqref{eqn:rossi-scalar} and~\eqref{eqn:rossi-torsion} yields
 \begin{align*}
  10(1-t^2)^2 & = 1 + 18t^2 + t^4 , & \pm 6(1-t^2)^2 = 10t(1+t^2) .
 \end{align*}
 This system also has no solution.
\end{example}

\section{Stability of the Tanaka--Webster scalar curvature}
\label{sec:stability}

The main goal of this section is to prove the following stability result for the MTW scalar curvature.
Our first step is to prove a splitting theorem for the space of smooth functions in terms of the linearization of the MTW scalar curvature.

\begin{lem}
 \label{splitting}
 Let $(M^{2n+1},J,\theta)$ be a compact, strictly pseudoconvex, partially integrable, pseudohermitian manifold.
 Then
 \begin{equation*}
  C^\infty(M) = \ker\Gamma \oplus \im DR .
 \end{equation*}
\end{lem}

\begin{proof}
 We first compute the leading-order term of $DR \circ \Gamma$ at $(J,\theta)$.
 To that end, we say that a differential operator $A \colon C^\infty(M) \to C^\infty(M)$ has order $k \in \mathbb{N}$ if it can be locally expressed as a sum
 \begin{equation*}
  A = \sum_{j=0}^k \sum_{i_1,\dotsc,i_j=1}^n c_{i_1\dotsm i_j}X_{i_1}\dotsm X_{i_j} ,
 \end{equation*}
 where $\{ X_1,\dotsc,X_n \}$ is a local frame for $\ker\theta$ and $c_{i_1\dotsm i_j}$ are smooth functions.
 For example, $DR \circ \Gamma$ has order $4$ and Equation~\eqref{eqn:matsumoto-commutator} implies that $\nabla_0$, the covariant derivative in the direction of the Reeb vector field, has order $2$.
 Given two differential operators $A$ and $B$, we write $A \equiv_k B$ if $A-B$ has order $k-1$.
 On the one hand, direct computation using Equations~\eqref{eqn:curv-commutator}, \eqref{eqn:linearization} and~\eqref{eqn:adjoint-linearization} yields
 \begin{align*}
  \MoveEqLeft (DR \circ \Gamma) f \equiv_4 \nabla^\alpha\nabla^\beta\nabla_\alpha\nabla_\beta + \nabla_\alpha\nabla_\beta\nabla^\alpha\nabla^\beta + (n+1)^2\Delta_b^2  \\
   & \equiv_4 \nabla^\alpha\nabla_\alpha\nabla^\beta\nabla_\beta + i\nabla^\alpha\nabla_0\nabla_\alpha + \nabla_\alpha\nabla^\alpha\nabla_\beta\nabla^\beta - i\nabla_\alpha\nabla_0\nabla^\alpha + (n+1)^2\Delta_b^2 .
 \end{align*}
 On the other hand, Equations~\eqref{eqn:matsumoto-commutator} and~\eqref{eqn:tor-commutator} imply that
 \begin{align*}
  u_\alpha{}^\alpha & = \frac{1}{2}\Delta_b u + \frac{ni}{2}u_0 , \\
  \Delta_b\nabla_0u & \equiv_3 \nabla_0\Delta_b u , \\
  \nabla^\alpha\nabla_0\nabla_\alpha u & \equiv_4 \nabla^\alpha\nabla_\alpha\nabla_0 u ,
 \end{align*}
 respectively.
 Therefore
 \begin{align*}
  (DR \circ \Gamma) f & \equiv_4 \frac{1}{2}\Real(\Delta_b+ni\nabla_0)^2 + \Real i(\Delta_b+ni\nabla_0)\nabla_0 + (n+1)^2\Delta_b^2 \\
   & \equiv_4 \frac{2n^2+4n+3}{2}\Delta_b^2 - \frac{n(n+2)}{2}\nabla_0^2 \\
   & \equiv_4 \frac{2n^2+4n+3}{2}(\Delta_b + a\nabla_0)(\Delta_b - a\nabla_0) ,
 \end{align*}
 where $a = \sqrt{\frac{n(n+2)}{2n^2+4n+3}}$.

 Now, since $a$ is real, $\Delta_b \pm a\nabla_0$ has a parametrix in the Heisenberg calculus, and hence is hypoelliptic~\cite{BealsGreiner}*{Theorem~18.4}.
 Since Heisenberg pseudodifferential operators are closed under composition~\cite{BealsGreiner}*{Theorem~14.1}, $DR \circ \Gamma$ has a parametrix, and hence is hypoelliptic.
 Therefore~\cite{BealsGreiner}*{Theorem~19.16} we have the $L^2$-orthogonal splitting $C^\infty(M) = \ker (DR \circ \Gamma) \oplus \im (DR \circ \Gamma)$.
 The final conclusion follows from the facts $\ker (DR \circ \Gamma) = \ker\Gamma$ and $\ker\Gamma\cap\im DR = \{0\}$.
\end{proof}

The stability of the MTW scalar curvature follows from the Implicit Function Theorem.

\begin{proof}[Proof of Theorem~\ref{thm1}]
 The proof of Lemma~\ref{tangent-space} implies that
 \begin{equation*}
  \Phi_{(J,\theta)}(E,u) := ( J\exp(-JE) , e^u\theta )
 \end{equation*}
 defines a smooth map $\Phi_{(J,\theta)} \colon T_{(J,\theta)}(\mathcal{C} \times \mathcal{K}) \to \mathcal{C} \times \mathcal{K}$.
 Direct computation implies that the differential $D\Phi_{(J,\theta)}$ at $(0,0)$ is the identity.
 In particular, there is a neighborhood $U \subset T_{(J,\theta)}(\mathcal{C} \times \mathcal{K})$ of the origin such that $\Phi_{(J,\theta)} \colon U \to \mathcal{C} \times \mathcal{K}$ is a diffeomorphism onto its image.

 Consider now the composition $R \circ \Phi_{(J,\theta)} \colon T_{(J,\theta)}(\mathcal{C} \times \mathcal{K}) \to C^\infty(M)$.
 Since $(M^{2n+1},J,\theta)$ is nonsingular, $D(R\circ\Phi_{(J,\theta)})_{(0,0)}$ is surjective.
 The conclusion now follows from the Implicit Function Theorem~\cite{Lang}*{Theorem~5.9}.
\end{proof}

The discussion of Section~\ref{section3} implies that Theorem~\ref{thm1} applies quite broadly.
For instance, combining Example~\ref{ex:rossi} with Theorem~\ref{thm1} shows that one can deform the Rossi sphere $(S^3,J_t,\theta)$, $0 < \lvert t \rvert < 1$ to prescribe the MTW scalar curvature close to $\frac{2(1+t^2)}{1-t^2}$.
More strikingly, any scalar-flat pseudohermitian manifold with nonvanishing torsion can be deformed to have arbitrary MTW scalar curvature.

\begin{cor}
 Let $(M^{2n+1},J,\theta)$ be a compact, strictly pseudoconvex, partially integrable, pseudohermitian manifold with vanishing MTW scalar curvature and nonvanishing torsion.
 For any $\varphi \in C^\infty(M)$, there is a pair $(\widetilde{J},\widetilde{\theta}) \in \mathcal{C} \times \mathcal{K}$ such that $R^{\widetilde{J},\widetilde{\theta}} = \varphi$.
\end{cor}

\begin{proof}
 Proposition~\ref{prop7} implies that $(M^{2n+1},J,\theta)$ is not $R$-singular.
 Pick $\varepsilon>0$ sufficiently small that Theorem~\ref{thm1} produces a pair $(J_1,\theta_1) \in \mathcal{C} \times \mathcal{K}$ such that $R^{J_1,\theta_1} = \varepsilon\varphi$.
 Then $R^{J_1,\varepsilon\theta_1} = \varphi$.
\end{proof}

A modification of the proof of Theorem~\ref{thm1} allows us to also prescribe the MTW scalar curvature when $\ker\Gamma=\mathbb{R}$, provided we choose a function with zero average.

\begin{cor}
 Let $(M^{2n+1},J,\theta)$ be a compact, strictly pseudoconvex, partially integrable, torsion-free, pseudohermitian manifold with vanishing MTW scalar curvature.
 Denote
 \begin{equation*}
  \Phi := \left\{ \varphi \in C^\infty(M) \mathrel{}:\mathrel{} \int_M \varphi \, \theta \wedge d\theta^n = 0 \right\} .
 \end{equation*}
 If $\varphi \in \Phi$, then there is a pair $(\widetilde{J},\widetilde{\theta}) \in \mathcal{C} \times \mathcal{K}$ such that $R^{\widetilde{J},\widetilde{\theta}} = \varphi$.
\end{cor}

\begin{proof}
 Note that $T_0\Phi = \Phi$ and that the $L^2$-orthogonal complement $\Phi^\perp$ is the space of constant functions on $M$.
 Proposition~\ref{prop0} and Lemma~\ref{splitting} then imply that $\im DR = \Phi$.
 As in the proof of Theorem~\ref{thm1}, the Implicit Function Theorem implies that there is an $\varepsilon>0$ sufficiently small such that there is a $(J_1,\theta_1) \in \mathcal{C} \times \mathcal{K}$ for which $R^{J_1,\theta_1} = \varepsilon\varphi$.
 Therefore $R^{J_1,\varepsilon\theta_1} = \varphi$.
\end{proof}

\section{Towards rigidity}
\label{sec:rigidity}

We conclude this paper by studying the second variation of the MTW scalar curvature of a compact, scalar-flat, torsion-free, integrable CR manifold.
This is a first step towards understanding the rigidity of the MTW scalar curvature of such manifolds.
We begin by computing the integral of the second variation of the MTW scalar curvature at such a manifold.
To that end, given $\tau,\sigma \in C^\infty(M;\xi^\ast)$, denote $\langle \tau, \sigma \rangle := 2\Real\tau^\alpha\sigma_\alpha$ and $\lvert\tau\rvert^2 := \langle \tau, \tau \rangle$.

\begin{prop}
 \label{second-variation}
 Let $(M^{2n+1},J,\theta)$ be a compact, integrable, torsion-free, pseudohermitian manifold with vanishing MTW scalar curvature.
 Then
 \begin{multline*}
  \int_M D^2R\bigl((E,u),(E,u)\bigr) \, \theta \wedge d\theta^n \\
   = \int_M \biggl[ 2\langle du, \epsilon \rangle - 2n\Real(iE^{\alpha\beta}\nabla_0E_{\alpha\beta}) - (n+1)(n+2)\lvert du \rvert^2  \biggr] \, \theta \wedge d\theta^n ,
 \end{multline*}
 where $\epsilon := 2\Real(i\nabla^\beta E_{\alpha\beta} \, \theta^\alpha)$.
\end{prop}

\begin{proof}
 Let $\tau \in C^\infty(M ; \xi^\ast)$ and denote $d^\ast\tau := -(\nabla^\alpha\tau_\alpha + \nabla^{\bar\alpha}\tau_{\bar\alpha})$.
 Then $d^\ast\tau$ is the $L^2$-adjoint of the projection $\pi_{\xi^\ast} \circ d$ of the exterior derivative on functions to $\xi^\ast$;
 i.e.
 \begin{equation*}
  \int_M ud^\ast\tau \, \theta \wedge d\theta^n = \int_M \langle du, \tau \rangle \, \theta \wedge d\theta^n .
 \end{equation*}
 Clearly $d^\ast\tau$ depends on both $J$ and $\theta$.

 Let $(J_t,\theta_t)$ be a curve in $\mathcal{C} \times \mathcal{K}$ with tangent vector $(E,u) \in T_{(J,\theta)}(\mathcal{C} \times \mathcal{K})$ at $t=0$.
 Let dots denote the first variation of a geometric quantity in the direction $(E,u)$.
 On the one hand, differentiating the identity $\int d^\ast\tau \, \theta \wedge d\theta^n = 0$ yields
 \begin{equation}
  \label{eqn:integrate-linearized-divergence}
  \int_M\dot d_t^\ast\tau \, \theta \wedge d\theta^n = -(n+1)\int_M \langle du , \tau \rangle \, \theta \wedge d\theta^n .
 \end{equation}
 On the other hand, Equations~\eqref{eqn:cr-linearize-R} and~\eqref{eqn:linearize-R-CR} imply that
 \begin{multline*}
  \int_M D^2R\bigl((E,u),(E,u)\bigr) \, \theta \wedge d\theta^n \\
   = \int_M \left( -\dot d^\ast\epsilon - 2n\Real\dot A^{\alpha\beta}E_{\alpha\beta} + (n+1)\dot d^\ast du - \dot Ru \right) \, \theta \wedge d\theta^n .
 \end{multline*}
 We deduce from Lemma~\ref{cr-linearization}, Proposition~\ref{prop1.3}, and Equation~\eqref{eqn:integrate-linearized-divergence} that
 \begin{align*}
  \int_M \dot d^\ast\epsilon \, \theta \wedge d\theta^n & = -(n+1)\int_M \langle du, \epsilon \rangle \, \theta \wedge d\theta^n , \\
  \int_M \Real \dot A^{\alpha\beta}E_{\alpha\beta} \, \theta \wedge d\theta^n & = \int_M \Real E^{\alpha\beta}(iu_{\alpha\beta} + i\nabla_0E_{\alpha\beta}) \, \theta \wedge d\theta^n , \\
  \int_M \Real \dot d^\ast du \, \theta \wedge d\theta^n & = -(n+1)\int_M \lvert du \rvert^2 \, \theta \wedge d\theta^n , \\
  \int_M \dot Ru \, \theta \wedge d\theta^n & = \int_M \left( -\langle du,\epsilon \rangle + (n+1)\lvert du \rvert^2 \right) \, \theta \wedge d\theta^n .
 \end{align*}
 The conclusion readily follows.
\end{proof}

Our partial infinitesimal rigidity result for the MTW scalar curvature follows:

\begin{proof}[Proof of Theorem~\ref{partial-rigidity}]
 Since $R^{J_0,\theta_0}=0$ and $R^{J_t,\theta_t}\geq0$, we see that
 \begin{equation*}
  0 \leq \left. \frac{\partial^2}{\partial t^2} \right|_{t=0} R^{J_t,\theta_t} = D^2R\bigl((E,u),(E,u)\bigr) + DR(\ddot J,\ddot u) .
 \end{equation*}
 On the one hand, Proposition~\ref{prop0} implies that
 \begin{equation*}
  \int_M DR(\ddot J,\ddot u) \, \theta \wedge d\theta^n = \int_M \langle \Gamma(1) , (\ddot J, \ddot u) \rangle \, \theta \wedge d\theta^n = 0 .
 \end{equation*}
 On the other hand, Proposition~\ref{second-variation} yields
 \begin{multline*}
  \int_M D^2R\bigl((E,u),(E,u)\bigr) \, \theta \wedge d\theta^n \\ = \int_M \Bigl( 2\langle du, \epsilon\rangle - 2\Real(niE^{\alpha\beta}\nabla_0E_{\alpha\beta})  - (n+1)(n+2)\lvert du \rvert^2 \Bigr) \, \theta \wedge d\theta^n .
 \end{multline*}
 Combining these observations yields
 \begin{equation}
  \label{eqn:second-variation-conclusion}
  0 \leq \int_M \Bigl( 2\langle du, \epsilon\rangle - 2\Real(niE^{\alpha\beta}\nabla_0E_{\alpha\beta})  - (n+1)(n+2)\lvert du \rvert^2 \Bigr) \, \theta \wedge d\theta^n
 \end{equation}
 with equality if and only if $\frac{\partial^2}{\partial t^2}\bigr|_{t=0} R^{J_t,\theta_t} = 0$.

 Since $(J_0,\theta_0)$ is flat and integrable, Lee's commutator identities~\cite{Lee1}*{Lemma~2.3} yield
 \begin{align*}
  \Real (niE^{\alpha\beta}\nabla_0E_{\alpha\beta}) & = \Real E^{\alpha\beta}(\nabla^\gamma\nabla_\gamma E_{\alpha\beta} - \nabla_\gamma\nabla^\gamma E_{\alpha\beta}) , \\
  \Real (niE^{\alpha\beta}\nabla_0E_{\alpha\beta}) & = \Real nE^{\alpha\beta}(\nabla^\gamma\nabla_\alpha E_{\gamma\beta} - \nabla_\alpha\nabla^\gamma E_{\gamma\beta}) .
 \end{align*}
 Afeltra et al.\ showed~\cite{Afeltra}*{Lemma~2.5} that, since each $J_t$ is integrable, $\nabla_{[\alpha} E_{\gamma]\beta} = 0$.
 Note that $\lvert\epsilon\rvert^2 = 2\lvert\nabla^\beta E_{\alpha\beta}\rvert^2$.
 Combining these with the Divergence Theorem and the previous display yields
 \begin{equation}
  \label{eqn:epsilon-equation}
  \frac{n}{2}\int_M \lvert\epsilon\rvert^2 \, \theta \wedge d\theta^n = \int_M \left( (n-1)\lvert\nabla_\gamma E_{\alpha\beta} \rvert^2 + \lvert\nabla_{\bar\gamma} E_{\alpha\beta} \rvert^2 \right) \, \theta \wedge d\theta^n .
 \end{equation}
 In particular,
 \begin{equation*}
  -2\int_M \Real(niE^{\alpha\beta}\nabla_0E_{\alpha\beta}) \, \theta \wedge d\theta^n = \int_M \left( \frac{n}{n-1}\lvert\epsilon\rvert^2 - \frac{2n}{n-1}\lvert \nabla_{\bar\gamma}E_{\alpha\beta} \rvert^2 \right) \, \theta \wedge d\theta^n .
 \end{equation*}
 Inequality~\eqref{eqn:second-variation-conclusion} and the Cauchy--Schwarz inequality then imply that
 \begin{align*}
  0 & \leq \int_M \Bigl( 2\langle du, \epsilon\rangle + \frac{n}{n-1}\lvert\epsilon\rvert^2 - \frac{2n}{n-1}\lvert\nabla_{\bar\gamma}E_{\alpha\beta}\rvert^2 - (n+1)(n+2)\lvert du \rvert^2 \Bigr) \, \theta \wedge d\theta^n \\
   & \leq \frac{n}{n-1}\int_M \left( \frac{n-1}{n(n+1)(n+2)} + 1 - C \right) \lvert\epsilon\rvert^2 \, \theta \wedge d\theta^n .
 \end{align*}
 Our assumption on $C$ implies that $\epsilon=0$.
 Combining this with Equation~\eqref{eqn:epsilon-equation} and Lee's commutator formulas~\cite{Lee1}*{Lemma~2.5} implies that $E$ is parallel.
 We then conclude from Inequality~\eqref{eqn:second-variation-conclusion} that $du=0$.
\end{proof}

\section*{Acknowledgements}

This work was initiated through conversations with Hung-Lin Chiu.
The authors would like to thank him for his various suggestions and comments in the early stages of work on this paper.

JSC acknowledges support from a Simons Foundation Collaboration Grant for Mathematicians, ID 524601.
PTH acknowledges support from 
the National Science and Technology Council (NSTC), Taiwan, with grant Number: 112-2115-M-032 -006 -MY2.

\bibliography{bib}

\end{document}